\documentclass{amsart}
\usepackage{amsmath,amssymb,amsthm,url}
\usepackage{hyperref}

\newtheorem{theorem}{Theorem}
\newtheorem{lemma}[theorem]{Lemma}

\theoremstyle{definition}

\newcommand{\one}{{\mathbf{1}}}

\newcommand{\A}{\mathrm A}

\newcommand{\C}{\mathrm C}
\newcommand{\D}{\mathrm D}

\newcommand{\V}{\mathrm V}

\newcommand{\cS}{{\mathcal{S}}}
\newcommand{\cP}{{\mathcal{P}}}

\newcommand{\ZZ}{\mathbb{Z}}

\newcommand{\p}{\wp}

\renewcommand{\wr}{\mathop{\rm wr}}

\newcommand{\dCn}{\mathrm{C}_{n}^{(2)}}

\newcommand{\Cos}{\mathrm{Cos}}
\newcommand{\Aut}{\mathrm{Aut}}

\newcommand{\soc}{\mathrm{soc}}
\newcommand{\Alt}{\mathrm{Alt}}

\newcommand{\PSL}{\mathrm{PSL}}

\newcommand{\dGm}{\overrightarrow{\Gamma}}

\newcommand{\beg}{\mathop{{\rm beg}}}
\newcommand{\inv}{\mathop{{\rm inv}}}
\newcommand{\half}{\frac{1}{2}}
\newcommand{\dC}[1]{{\rm C}_{#1}^{(2)}}

\title[Half-arc-transitive graphs]{Smallest tetravalent half-arc-transitive graphs with the vertex-stabiliser isomorphic to the dihedral group of order $8$}

\author[P. Poto\v{c}nik]{Primo\v{z} Poto\v{c}nik}
\address{Primo\v{z} Poto\v{c}nik,\newline Faculty of Mathematics and Physics,
University of Ljubljana, Slovenia; \newline
also affiliated with\newline
 IAM, University of Primorska, Koper, Slovenia\newline
 and\newline
Institute of Mathematics, Physics, and
  Mechanics, Ljubljana, Slovenia}\email{primoz.potocnik@fmf.uni-lj.si}

\author[R. Po\v{z}ar]{Rok Po\v{z}ar}
\address{Rok Po\v{z}ar, FAMNIT,\newline
University of Primorska, Koper, Slovenia}\email{pozar.rok@gmail.com}

\thanks{Supported in part by Slovenian Research Agency, projects L1--4292, J1-5433, J1-6720, P1-0294 and P1-0285}

\subjclass[2000]{20B25}
\keywords{graph, tetravalent, vertex-transitive, edge-transitive} 

\begin{document}

\begin{abstract}
A connected graph whose automorphism group acts transitively on the edges and vertices, but not on the set of ordered
pairs of adjacent vertices of the graph
is called half-arc-transitive. It is well known that the valence of a half-arc-transitive graph is even and at least four.
Several infinite families of half-arc-transitive graphs of valence four are known, however, in all except four of the known
specimens, the vertex-stabiliser in the automorphism group is abelian. The first
example of a half-arc-transitive graph of valence four and with a non-abelian vertex-stabiliser was described in
[Conder and Maru\v{s}i\v{c}, {\em A tetravalent half-arc-transitive graph with non-abelian vertex stabilizer}, J.~Combin.~Theory Ser.~B \textbf{88} (2003) 67--76]. This example has $10752$ vertices 
and vertex-stabiliser isomorphic to the dihedral group of order $8$.
In this paper, we show that no such graphs of smaller order exist, thus answering a frequently asked question.
\end{abstract}

\maketitle

\section{Introduction}

Let $\Gamma$ be a connected finite graph and $G$ a group of automorphisms of $\Gamma$.
If $G$ acts transitively on the set of {\em vertices}, {\em edges} or {\em arcs} of the graph
(an {\em arc} is an ordered pair of adjacent vertices),
 then $\Gamma$ is said to be {\em $G$-vertex-transitive}, {\em $G$-edge-transitive} or 
{\em $G$-arc-transitive}, respectively. Furthermore, if $G$ acts transitively on the vertices, edges, but not
arcs of the graph, then $\Gamma$ is {\em $(G,\half)$-arc-transitive}. A graph $\Gamma$ is {\em $\half$-arc-transitive}
if it is $(G,\half)$-arc-transitive for $G=\Aut(\Gamma)$.

While graphs that are $(G,\half)$-arc-transitive for some group of automorphisms $G$ are rather easy to find (for example,
every cycle is such a graph),  $\half$-arc-transitive graphs are considerably more elusive and the question of their existence has
been posed as an open problem by Tutte \cite{Tutte} and answered affirmatively by Bouwer \cite{Bouwer} a few years later.
Since cycles are arc-transitive and every $(G,\half)$-arc-transitive graph
has to have even valence (as observed already in \cite{Tutte}),
 the smallest admissible valence for a $\half$-arc-transitive graph is $4$; and indeed, the smallest $\half$-arc-transitive
graph is tetravalent and of order $27$. The tetravalent $\half$-arc-transitive graphs (and half-arc-transitive graphs in general)
 were much studied by many authors
from different points of view, ranging from purely combinatorial \cite{HajKutMar,Mar98,MarPra,sajna,Spa08,Spa09,steve},
 geometrical \cite{MarNedMaps}, 
to permutation group theoretical \cite{PraPreprint,LiMar,MalMar99}
 and abstract group theoretical \cite{MarNed3,PVdigraphs,genlost}.


As observed by Maru\v{s}i\v{c} and Nedela \cite{MarNed3}, a group theoretical result of Glauberman \cite{Glaub} implies that
the vertex-stabiliser $G_v$ in a tetravalent $(G,\half)$-arc-transitive graph is a group of order $2^s$ for some $s\ge 1$,
of nilpotency class at most $2$, generated by $s$ involutions, and satisfying certain addition conditions
 (see \cite[Theorem 1.1]{MarNed3} for details and \cite[Theorem 1.2]{PVdigraphs} for a generalisation of this result to graphs of larger valence).

While each of the $2$-groups described in \cite[Theorem 1.1]{MarNed3}
can indeed occur as the vertex-stabiliser $G_v$ in a tetravalent $(G,\half)$-arc-transitive graph,
it remains an open problem which of the groups of \cite[Theorem 1.1]{MarNed3} can occur
as the vertex-stabiliser in the full automorphism group of a tetravalent $\half$-arc-transitive graph.
This problem was resolved for the case of abelian vertex-stabilisers in \cite{DM05},
where for every positive integer $s$, 
a tetravalent $\half$-arc-transitive graph with $G_v$ isomorphic to $\ZZ_2^s$ was
constructed.

Non-abelian vertex-stabilisers seem to be much more elusive in this respect.
The first example of a tetravalent $\half$-arc-transitive graph with a non-abelian vertex-stabiliser has been constructed
by Conder and Maru\v{s}i\v{c} \cite{ConMar}. Their example has order $10752$ and vertex-stabiliser isomorphic to
the dihedral group $\D_4$ of order $8$. In \cite{ConPotSpa}, two more examples with $\Aut(\Gamma)_v \cong \D_4$
were found (another one of order $10752$ and one of order $21870$),
and also the first known example of a tetravalent $\half$-arc-transitive graph with a non-abelian vertex-stabiliser of order $16$;
the latter having order $90\cdot 3^{10}$. To the best of our knowledge, these four graphs are the only known tetravalent
$\half$-arc-transitive graphs with a non-abelian vertex-stabiliser.

Given that the order of the smallest  tetravalent 
$\half$-arc-transitive graph with the stabiliser isomorphic to $\D_4$, found by Conder and Maru\v{s}i\v{c}
 more than a decade ago,  is rather large,
the question of existence of a smaller specimen of this family has often been raised
by experts in this area. Perhaps surprisingly, as we prove in this paper, this question has a negative answer:

\begin{theorem}
\label{the:main}
There are no tetravalent $\half$-arc-transitive graphs  of order
less than $10752$
with the vertex-stabiliser isomorphic to $\D_4$, and there are precisely two such examples of order $10752$.
\end{theorem}

The proof of this theorem is a combination of a few theoretical results (see Lemma~\ref{lem:non-solv} and Lemma~\ref{lem:quoN}) 
relying heavily on a surprising and deep result  of Praeger \cite{PraHATD} (see Theorem~\ref{the:Pra}),
and computer assisted computations, based on the theory of lifting automorphisms along covering projection of graphs.
Using this approach, we were able to produce a complete list of all tetravalent $(G,\half)$-arc-transitive graphs 
of order at most $10752$ with $G_v\cong \D_4$. It transpires that there are precisely $564$ such graphs
(see Theorem~\ref{the:census}). The {\sc Magma} \cite{Magma} code that generates these graphs can be found at \cite{PP14}.
Theorem~\ref{the:main} follows by inspection of the full automorphism groups of these $564$ graphs -- it turns out that all but
two of them have automorphism group acting transitively on the arcs.


We thank Pablo Spiga and Gabriel Verret
for reading the first draft of this paper and making some valuable comments.

\section{Preliminaries}

\subsection{Concerning graphs}
\label{sec:graphs}

Perhaps the most accepted definition of a graph is that of a {\em simple graph}, an object determined by
its {\em vertex-set} $V$ and its {\em edge-set} $E$ satisfying $E\subseteq \{e : e \subseteq V, |e|=2\}$.
Even though we are mainly interested in simple graph 
(especially since non-simple tetravalent edge-transitive graphs are
easily classified; see Lemma~\ref{lem:simple}), 
allowing graphs to have loops, multiple edges and semiedges
proves to be most convenient in the proofs; in particular,
a most useful Lemma~\ref{lem:nq} does not hold if one insists on allowing
simple graphs only.

In what follows, we will therefore adopt a slightly more general model of a graph. As in, say, \cite{ElAbCov},
for us, a {\em graph} will be an ordered $4$-tuple $(D,V; \beg,\inv)$ where
$D$ and $V \neq \emptyset$ are disjoint finite sets of {\em darts}
and {\em vertices}, respectively, $\beg: D \to V$ is a mapping
which assigns to each dart $x$ its {\em initial vertex}
$\beg\,x$, and $\inv: D \to D$ is an involution which interchanges
every dart $x$ with its {\em inverse dart}, also denoted by $x^{-1}$.
If $\Gamma$ is a graph, then we let $\D(\Gamma)$ and $\V(\Gamma)$
denote its dart-set and its vertex-set, respectively.

The cardinality $|V|$ of $V$ is called the {\em order} of $\Gamma$.
The {\em neighbourhood} of a vertex $v$, denoted $\Gamma(v)$,
 is the set consisting of all the darts $x$ with $\beg(x)=v$, and the cardinality of $\Gamma(v)$
is called the {\em valence} of $v$.
 A graph is {\em tetravalent} if all of its vertices have valence $4$. 

The orbits of $\inv$ are called {\em edges}.
The edge containing a dart $x$ is called a {\em semiedge} if $\inv\,x = x$,
a {\em loop} if $\inv\,x \neq x$ while
$\beg\,(x^{-1}) = \beg\,x$,
and  is called  a {\em link} if $\inv\,x \neq x$ and $\beg\,(x^{-1}) \neq \beg\,x$.
The {\em endvertices of an edge} are the initial vertices of the darts contained in the edge.
Two links are {\em parallel} if they have the same endvertices.

A graph with no semiedges, no loops and no parallel links
is called a {\em simple graph} and can be given uniquely in the usual manner,
by its vertex-set and edge-set. Conversely, any simple graph, given in terms of its vertex-set $V$ and edge-set $E$
can be easily viewed as the graph $(D,V; \beg,\inv)$, where
$D=\{(u,v) \mid \{u,v\} \in E\}$,
$\inv (u,v) = (v,u)$ and  $\beg(u,v) = u$ for any $(u,v) \in D$.

Let $\Gamma= (D,V; \beg,\inv)$ and $\Gamma'= (D',V'; \beg',\inv')$ be two graphs.
A {\em morphism of graphs}, $f \colon \Gamma \to \Gamma'$,
is a function $f \colon V \cup D \to V' \cup D'$
such that $f(V) \subseteq V'$, $f(D) \subseteq D'$,
$f\circ \beg = \beg' \circ f$ and $f \circ \inv = \inv' \circ f$.
A graph morphism is an {\em epimorphism} ({\em automorphism}) if it is
a surjection (bijection, respectively). 
If $g$ and $h$ are automorphisms of $\Gamma$ and $x$ is a dart or a vertex of $\Gamma$,
then we denote
the $g$-image of $x$ by $x^g$ and multiply the automorphisms $g$ and $h$
so that $(v^{gh}) = (v^g)^h$.
The set of automorphisms of $\Gamma$ together with the above product forms a group,
called the {\em automorphism group} of $\Gamma$ and  denoted by $\Aut(\Gamma)$. 

The graph $\Gamma$ is called {\em vertex-transitive}
({\em dart-transitive}, respectively), provided that $\Aut(\Gamma)$ acts transitively on
vertices (darts, respectively) of $X$. Note that in the context of simple graphs, a {\em dart} is
often called an {\em arc} of a graph; hence the term {\em arc-transitive} is also used as a synonym for dart-transitive.

An example of non-simple tetravalent edge-transitive (indeed, dart-transitive) graph is a {\em doubled cycle} $\dC{n}$,
which can be viewed as the usual cycle on $n$ vertices, but with each edge doubled. Formally, 
the graph $\dC{n}$ is $(V,D,\beg,\inv)$, where
$V=\ZZ_n$,  $D=\ZZ_n\times\ZZ_2\times\ZZ_2$, 
$\beg(i,\epsilon,j) = i$,
and $\inv(i,\epsilon,j) = (i+(-1)^\epsilon,1-\epsilon,j)$ for every $i\in \ZZ_n$ and $\epsilon, j\in \ZZ_2$.
If $n=1$, then $\dCn$ is the graph with a single vertex and two loops attached to it, while in the case $n=2$, $\dCn$ consists
of two vertices and $4$ parallel links connecting them.
We state the following lemma without a proof:

\begin{lemma}
\label{lem:simple}
A connected tetravalent edge-transitive graph is either simple, isomorphic to $\dC{n}$ for some positive integer $n$, or
isomorphic to the graph with a single vertex and four semiedges attached to it.
\end{lemma}

\subsection{Concerning quotients and covers}

This section summarises some of the facts about quotients and covers that can be derived easily from \cite{MalNedSko}
or \cite{ElAbCov}.

An epimorphism $\p \colon \Gamma \to \Lambda$ is a {\em covering projection}
provided that the restriction $\p_v \colon \Gamma(v) \to \Lambda(\p(v))$
of $\p$ to the neighbourhood of $v$ is bijective for every $v\in\V(\Gamma)$.
For simplicity, we shall also require both $\Gamma$ and $\Lambda$ to be connected.
The preimage $\p^{-1}(x)$ of a vertex or a dart $x$ of
$\Lambda$ is called a {\em fibre} of the covering projection $\p$ 
and the group of all automorphisms of $\Gamma$ that preserve each fibre set-wise
is called the {\em group of covering transformations}. 
If the latter is transitive on each fibre, then it acts regularly on each fibre, and
the covering projection is {\em regular}. A regular covering projection with the group of
covering transformations being elementary abelian or solvable is itself called {\em elementary abelian} or {\em solvable}, respectively.

If $\p \colon \Gamma \to \Lambda$ is a covering projection, $g\in \Aut(\Gamma)$
and $h\in \Aut(\Lambda)$ such that $\p(x^g) = \p(x)^h$ for every $x\in \V(\Gamma)\cup \D(\Gamma)$,
then we say that $g$ {\em projects} along $\p$ to $h$ and that $h$ {\em lifts} along $\p$ to $g$.
A subgroup $H\le\Aut(\Lambda)$ {\em lifts} along $\p$ if and only
if each of its elements lifts; in this case, the set of all the lifts of elements of $H$ forms a group,
called {\em the lift of $H$ along $\p$}. If $H$ lifts along $\p$, then we also say that $\p$ is {\em $H$-admissible}.
Note that the group of covering transformations is nothing
but the lift of the trivial group of automorphisms of $\Lambda$.

The composition $\p_2 \circ \p_1$ of two covering projections $\p_1 \colon \Gamma_1 \to \Gamma_2$
and $\p_2\colon \Gamma_2\to \Gamma_3$ is again a covering projection. If both $\p_1$
and $\p_2$ are regular, so is the composition $\p_2 \circ \p_1$.
If a group $H\le \Aut(\Gamma_3)$ lifts along $\p_2$ to some group that lifts further along  $\p_1$,
 and if neither of $\p_1$ and $\p_2$ is a graph-isomorphism,
then we say that the covering projection $\p_2 \circ \p_1$ {\em splits with respect to $H$}. An $H$-admissible
covering projection that does not split with respect to $H$ is said to be a {\em minimal $H$-admissible covering projection}.

Inspired by a paper of Lorimer \cite{lor} on arc-transitive graphs of prime valence,
Praeger \cite{Pra2} introduced the concept of {\em normal quotients}, 
which has now become a standard tool in studying symmetries of graphs. 
Here we adapt this concept slightly so as to fit into the setting of graphs admitting loops, multiple edges and semiedges.
This adaptation will prove most useful in the proof of our main result.

Let $\Gamma$ be a graph and let $N \leq \Aut(\Gamma)$.
Let $D_N = \{x^N : x\in \D(\Gamma)\}$ and $V_N=\{v^N : v\in \V(\Gamma)\}$ 
 denote the sets of $N$-orbits on the darts and vertices of $\Gamma$, respectively.
Further, let $\beg_N \colon D_N \to V_N$ and $\inv_N \colon D_N \to V_N$ be defined by
 $\inv_N(x^N)=\beg(x)^N$ and  $\inv_N(x^N)=\inv(x)^N$.
 The quadruple $\Gamma/N= (V_N,D_N,\beg_N,\inv_N)$ is then called the {\em quotient graph}
  of $\Gamma$ with respect to $N$.
 Note that the mapping $\p_N \colon \Gamma \to \Gamma/N$,
 defined by $\p_N(x) = x^N$ for every  $x\in \V(\Gamma) \cup \D(\Gamma)$, 
 is a graph epimorphism,
called the {\em normal quotient projection relative to} $N$.

If $N$ is a normal subgroup of a group $G\le \Aut(\Gamma)$,
then there is an obvious, but not necessarily faithful action of 
the quotient group $G/N$ on the graph $\Gamma/N$.
Note also that if $G$ acts transitively on vertices, darts or edges of $\Gamma$, then
so does $G/N$ on $\Gamma/N$.

A natural question that arises at this point is under what circumstances the 
quotient projection is a covering projection. The following lemma
provides two very simple (and easy-to-prove) sufficient and necessary conditions for this to happen.
Note that the analogue result does not hold if one insists on considering simple graphs only and
the concept of normal quotients as defined in \cite{Pra2}.

\begin{lemma}
\label{lem:nq}
Let $\Gamma$ be a connected graph, let $N\le \Aut(\Gamma)$ and 
let $\p \colon \Gamma \to \Gamma/N$ be the corresponding quotient projection.
Then the following statements are equivalent:
\begin{enumerate}
\item 
$N$ is semiregular (that is, $N_v=\one$ for every $v\in\V(\Gamma)$);
\item
for every $v\in V(\Gamma)$, the valence of $v$ in $\Gamma$ equals the valence of $\p(v)$ in $\Gamma/N$;
\item
$\p$ is a regular covering projection and $N$ is the group of covering transformations of $\p$.
\end{enumerate}
\end{lemma}

We finish the section by presenting a basic but very useful result about regular covering projections.

\begin{lemma}
\label{lem:isoGv}
Let $\Gamma$ be a graph, let $G\le\Aut(\Gamma)$ and let $N$ be a normal subgroup of $G$. 
If the quotient projection $\p\colon \Gamma \to \Gamma/N$ is a covering projection,
then the action of $G/N$ on $\V(\Gamma/N)\cup \D(\Gamma/N)$ is faithful, 
and  the stabilisers $G_v$ and $(G/N)_{v^N}$ are isomorphic for every $v\in \V(\Gamma)$.
Moreover, the group $G$ projects to $G/N$ and $G$ is the lift of $G/N$ along $\p$.
The group $G$ is transitive on vertices, darts or edges of $\Gamma$ if and only if $G/N$ 
 is transitive on vertices, darts or edges of $\Gamma/N$. In particular, $\Gamma$ is $(G,\half)$-arc-transitive if and only if $\Gamma/N$ is $(G/N,\half)$-arc-transitive.
\end{lemma}

\section{Proof of the main result}

We begin this section by pointing out an intimate relationship between tetravalent $(G,\half)$-arc-transitive graphs
and $G$-arc-transitive digraphs of valence $2$. Let us first explain, how we understand the notion of a {\em digraph}.

Here, a {\em digraph} is a pair $\dGm = (\Gamma,\Delta)$, where $\Gamma$ is a graph (called the {\em underlying graph} of $\dGm$)
and  $\Delta \subseteq \D(\Gamma)$ such that $|\Delta \cap \{x,x^{-1}\}| = 1$ for every dart $x\in \D(\Gamma)$;
 the set $\Delta$ is then called the {\em arc-set} of $\dGm$, and denoted $\A(\dGm)$.
An {\em automorphism} of the digraph $\dGm$ is an automorphism of $\Gamma$ that preserves $\Delta$ set-wise,
and a group $G$ of automorphisms of $\dGm$ is {\em arc-transitive} provided that $G$ acts transitively on $\Delta$.
The digraph $\dGm$ is $k$-valent if every vertex is the initial vertex of precisely $k$ darts in $\Delta$ as well as $k$ darts
in $\D(\Gamma) \setminus \Delta$.

If $\Gamma$ is a $(G,\half)$-arc-transitive graph of valence $2k$, then $G$ has two orbits on $\D(\Gamma)$, 
each orbit containing precisely one dart of each edge. If $\Delta$ is one of these orbits,
then $(\Gamma,\Delta)$ is a $G$-arc-transitive digraph of valence $k$; we shall say that
$(\Gamma,\Delta)$ {\em arises} form the $(G,\half)$-arc-transitive graph $\Gamma$.
Conversely, if $\dGm$ is a $G$-arc-transitive digraph of valence $k$,
then its underlying graph is $(G,\half)$-arc-transitive and of valence $2k$.
In this sense, the study of tetravalent $(G,\half)$-arc-transitive graphs is equivalent to the study of $2$-valent
$G$-arc-transitive digraphs.

The following theorem represents one of the main ingredients of the proof of our main result.
Its proof in the case of simple graphs follows directly from an astonishing characterisation
of $G$-arc-transitive digraphs of prime valency for which the group $G$ contains a non-semiregular
abelian normal subgroup, proved by Praeger in \cite{PraHATD}.

\begin{theorem}
\label{the:Pra}
Let $\Gamma$ be a connected tetravalent $(G,\half)$-arc-transitive graph.
If $G$ contains an abelian normal subgroup $N$ such
that $N_v\not= \one$ for some vertex $v$, then $G$ is solvable and $G_v$ is elementary abelian.
\end{theorem}

\begin{proof}
Let $\dGm$ be the $G$-arc-transitive $2$-valent digraph arising from $(\Gamma,G)$.

Suppose first that $\Gamma$ is not simple. Then by Lemma~\ref{lem:simple}, $\Gamma \cong\dC{n}$
 for some $n\ge 1$ (note that the graphs with semiedges do not allow
$\half$-arc-transitive groups of automorphisms). If $n=1$, then 
$\Aut(\dGm) =\Aut(\dGm)_v \cong \ZZ_2$.
Similarly, if $n=2$, then 
$\Aut(\dGm) \cong \D_4$ and
$\Aut(\dGm)_v \cong \ZZ_2^2$. In both cases,
the automorphism group is solvable and the vertex-stabiliser is elementary abelian, as claimed.
We may thus assume that $n\ge 3$. Now consider a pair $\{e,e'\}$ of parallel links of $\dC{n}$
and the pair $\{x,x'\}$ of arcs 
of $\dGm$ lying on these two links. If $x$ and $x'$ have the same initial vertex,
then this is the case for every parallel pair of links, implying that
$\dGm$ can be viewed as obtained from the doubled cycle $\dC{n}$ by orienting all the edges consistently 
in the clock-wise direction.
Note that $\Aut(\dGm)_v \cong \ZZ_2^n$ and $\Aut(\dGm) \cong \ZZ_2^n \rtimes \C_n$ in this case, implying that $G$ is solvable and $G_v$ is elementary abelian, as claimed.
On the other hand, if $x$ and $x'$ have distinct initial vertices, then it can be easily
seen that $\Aut(\dGm)$ acts faithfully on $\V(\dGm)$, is isomorphic to $\D_{n}$, and that $\Aut(\dGm)_v \cong \ZZ_2$.

We may thus assume that $\Gamma$ is simple. By
 \cite[Theorem 2.9]{PraHATD} it follows that $\dGm$ is isomorphic to a certain
 digraph $C_n(2,t)$ (for some integer $t$), defined in \cite[Definition 2.6]{PraHATD}. 
 However, by \cite[Theorem 2.8(c)]{PraHATD},
 the automorphism group of such a digraph is solvable and the vertex-stabiliser
  is elementary abelian, concluding the proof of the theorem.
\end{proof}

\begin{lemma}
\label{lem:non-solv}
If $\Gamma$ is a tetravalent $(G,\frac{1}{2})$-arc-transitive graph with the vertex-stabiliser $G_v$ being non-abelian, then $G$ is non-solvable.
\end{lemma}

\begin{proof}
Suppose that the lemma is false and let $\Gamma$ be a minimal counter-example, that is, 
$\Gamma$ is a graph of smallest order admitting a solvable
 $\frac{1}{2}$-arc-transitive group of automorphisms $G$ with $G_v$ non-abelian. 

Let $N$ be a minimal normal subgroup of $G$. Since $G$ is solvable, $N$ is elementary abelian. 
Since $G_v$ is non-abelian, Theorem~\ref{the:Pra} implies that $N_v = \one$.
But then, by Lemma~\ref{lem:nq} and Lemma~\ref{lem:isoGv},  
$\Gamma/N$ is a tetravalent $(G/N,\half)$-arc-transitive graph with 
$G/N$ solvable and the vertex-stabiliser in $G/N$ being isomorphic to $G_v$, and thus non-abelian.
However, this contradicts the minimality of $\Gamma$, and thus proves the lemma.
\end{proof}

The subgroup $N$ of a finite group $G$ generated by all solvable normal subgroups of $G$ is
itself solvable and normal in $G$ and is called the {\em maximal solvable normal subgroup} of $G$.
Note that $G/N$ then contains no non-trivial solvable normal subgroups. Such groups are called
{\em semisimple} (or also groups with a {\em trivial solvable radical}).

\begin{lemma}
\label{lem:quoN}
Let $\Gamma$ be a tetravalent $(G,\half)$-arc-transitive graph with $G$ non-solvable and let $N$ be the maximal solvable normal subgroup of $G$. Then the quotient projection $\Gamma \to \Gamma/N$ is a solvable regular covering projection,
 with $\Gamma/N$ a simple $(G/N,\frac{1}{2})$-arc-transitive graph, along which the semisimple group $G/N$ lifts.
\end{lemma}

\begin{proof}
Observe first that all we need to show is that $N$ is semiregular. 
The fact that $\Gamma/N$ is simple then follows from Lemma~\ref{lem:simple} and
the fact that $G/N$ is non-solvable, and the rest follows from Lemma~\ref{lem:nq}
and Lemma~\ref{lem:isoGv}.

Suppose now that the lemma is false and let $\Gamma$ be a minimal counterexample,
(in terms of $|\V(\Gamma)|$). Let $G$ be the corresponding non-solvable $\half$-arc-transitive group
of automorphisms of $\Gamma$ and $N$ the maximal solvable normal subgroup of $G$ such that $N_v\not = \one$.

Since $N$ is solvable and non-trivial, the socle $S=\soc(N)$ is non-trivial and abelian.
Since $S$ is characteristic in $N$, it is normal in $G$.
Since $G$ is non-solvable, Theorem~\ref{the:Pra} implies that $S_v = \one$.
By Lemma~\ref{lem:nq}, the quotient projection $\Gamma \to \Gamma/S$ is then a covering projection,
and by Lemma~\ref {lem:isoGv},
$\Gamma/S$ is a $(G/S,\frac{1}{2})$-arc-transitive graph.
 Note that $N/S$ is the maximal solvable subgroup of $G/S$.
Since $G/S$ is non-solvable, the minimality of $\Gamma$ implies that 
$N/S$ acts semiregularly on $\V(\Gamma/S)$.
But then, by Lemma~\ref{lem:isoGv}, $N_v = \one$, contradicting our assumption on $N$.
\end{proof}

\subsection{Constructing a census of $(G,\half)$-arc-transitive graphs with $G_v\cong \D_4$}

To shorten the text, we will call
 a pair $(\Gamma,G)$, where $\Gamma$ is a connected tetravalent 
$(G,\half)$-arc-transitive graph with the vertex-stabiliser isomorphic to the dihedral
group $\D_4$, {\em relevant}. Observe that if $(\Gamma,G)$ is relevant, then
$|G| = |G_v|\, |\V(\Gamma)| = 8 |\V(\Gamma)|$.

Lemmas~\ref{lem:non-solv} and \ref{lem:quoN} suggest a strategy for constructing all relevant pairs $(\Gamma,G)$
with $\Gamma$ of order at most $M$ (for a given fixed integer $M$).

\begin{enumerate}
\item Find the set $\cS$ of all semisimple groups of order at most $8M$.
\item Find the set $\cP_0$ of all relevant pairs $(\Gamma,G)$ with $G$ isomorphic
to some group in $\cS$.  
\item For every relevant pair $(\Gamma,G) \in \cP_0$, find all solvable regular covering projections
$\p\colon \tilde{\Gamma} \to \Gamma$ with $|\V(\tilde{\Gamma})|\le M$ along which $G$ lifts, and compute the lift $\tilde{G}$ of $G$ along $\p$. Let $\tilde{\cP}$ denote the set of all pairs
$(\tilde{\Gamma},\tilde{G})$ obtained in this way. 
\end{enumerate}

The union of  the sets $\cP_0$ and $\tilde{\cP}$, is then the set of all the relevant
pairs $(\Gamma,G)$ with $|\V(\Gamma)|\le M$.
Let us now discuss each step of this approach.

Step (1) is easy if one has a database of all simple groups of order at most $8M$ available. 
Namely, if  $G$ is a semisimple group, its socle $\soc(G)$ is isomorphic to a group
$T_1^{\alpha_1} \times  \ldots \times T_k^{\alpha_k}$ for some
pairwise non-isomorphic non-abelian simple groups $T_i$ and  positive integers $\alpha_i$.
Moreover, $G$ acts by conjugation upon $\soc(G)$ faithfully and thus embeds into $\Aut(\soc(G))$.
In this sense, we have $\soc(G)\le G \le \Aut(\soc(G))$.

For example, if $M=10752$, then, using the database of simple groups in {\sc Magma}, one can
easily show that there are precisely $100$ semisimple groups of order at most $8M$.

One of the possible approaches to Step (2) relies on the classification of
the ``universal groups'' acting half-arc-transitively on an infinite tetravalent tree.
These groups were determined in \cite{MarNed3} (see also \cite{PVdigraphs} 
and \cite[Section 3.2 and Table 1]{PSVdigraphs}).
In short, if follows from these results that every relevant pair arises
from the tetravalent infinite tree ${\rm T}_4$ and a group
$$
 U \cong \langle a,b,c,g \mid a^2,b^2,c^2,g^{-1}agb,g^{-1}bgc, (ab)^2, (bc)^2, (ac)^2b \rangle
$$
acting $\half$-arc-transitively on ${\rm T}_4$. Here
the subgroup $\langle a,b,c \rangle$, which is isomorphic to $\D_4$, corresponds to the stabiliser $U_{\tilde{v}}$
of a vertex $\tilde{v}\in \V({\rm T}_4)$
and $g$ to an automorphism of ${\rm T}_4$  mapping $\tilde{v}$ to a neighbour of $\tilde{v}$.
(Note that up to conjugacy in $\Aut({\rm T}_4)$, the above group is the unique 
$\half$-arc-transitive subgroup of $\Aut({\rm T}_4)$
satisfying $U_{\tilde{v}} \cong \D_4$.)

Now, for  every relevant pair $(\Gamma,G)$, 
there exists an epimorphism $f\colon U \to G$
such that for some vertex $v\in \V(\Gamma)$, we have $G_v = f(\langle a,b,c \rangle)$ and
$f(g)$ mapping  $v$ to a neighbour of $v$. In particular, if we know what $f$ is, then
the graph $\Gamma$ can be reconstructed as the coset graph $\Cos(G,f(\langle a,b,c\rangle),f(g))$
(see \cite[Section 3]{PSVdigraphs} for more information).

Step (2) thus amounts to finding all epimorphisms from $U$ to each of the $100$ semisimple groups of order at most $8\cdot 10752$.
This was easily done by the algorithms implemented in {\sc Magma} (here the fact that $U$ can be generated by two elements,
$a$ and $g$, is very helpful). The computations
 yield the set $\cP_0$ of $16$ relevant pairs $(\Gamma,G)$ with $G$ semisimple.
 The properties of these pairs are summarized  in Table~\ref{table:base}.
\begin{table}[!ht]\footnotesize
\caption{Properties of relevant pairs $(\Gamma,G)$ of the set $\cP_0$.} 
\centering 
\begin{tabular}{c  c  c  c  c  c} 
\hline\hline 
\\[-2ex]
ID & $|\Gamma|$ & $|\Aut(\Gamma)|$ & $\soc(G)$ & $|G|$ \\ [0.5ex] 
\hline 
\\[-2ex]
1 &  42 & 672  & $\PSL_2(7)$ & 336\\
2 &  90 &  2880 & $\Alt_6$ & 720\\
3 &  90 & 2880 & $\Alt_6$ & 720\\
4 &  306 & 4896 & $\PSL_2(17)$ & 2448\\
5 &  702 &  22464 & $\PSL_3(3)$ & 5616\\
6 &  756 &  12096 & $\textrm{U}_3(3)$ & 6048\\
7 &  1404 & 44928 & $\PSL_3(3)$ & 11232\\
8 &  1518 & 24288 & $\PSL_2(23)$ & 12144\\
9 &  1860 & 29760 & $\PSL_2(31)$ & 14880\\
10 &  1950 & 62400 & $\PSL_2(25)$ & 15600\\
11 &  1950 & 62400 & $\PSL_2(25)$ & 15600\\
12 &  5040 & 80640 & $\Alt_8$ & 40320\\
13 &  6486 & 103776 & $\PSL_2(47)$ & 51888\\
14 &  7056 & 225792 & $\PSL_2(7) \times \PSL_2(7)$ & 56448\\
15 &  7056 & 225792 & $\PSL_2(7) \times \PSL_2(7)$ & 56448\\
16 &  8610 & 137760 & $\PSL_2(41)$ & 68880\\\hline
\end{tabular}
\label{table:base} 
\end{table}

Finally, Step (3) relies on computing solvable regular covering projections along which a given group lifts, up to a prescribed order $M$ of the respective covering graphs. Such an algorithm has been developed and implemented in {\sc Magma} by the second author; it is described in detail in \cite{Poz14}, and is based on the theory of elementary abelian regular covering projections, presented in \cite{ElAbCov}. 
The essence of this algorithm relies on the fact that for every solvable
$G$-admissible regular covering projection $\p \colon \Gamma \to \Lambda$
there exists a sequence of elementary abelian regular
covering projections $\p_i \colon \Gamma_{i} \to \Gamma_{i-1}$ for $i\in \{1,\ldots, k\}$, and a sequence of
groups $G_{i} \le \Aut(\Gamma_{i})$ for $i\in \{0,\ldots, k\}$, such that the following hold:
\begin{enumerate}
\setlength{\itemsep=0pt}
\item 
 $\p = \p_1 \circ  \ldots \circ \p_k$ (in particular, $\Gamma_0 = \Lambda$ and $\Gamma_k=\Gamma$);
 \item
 $G_0 = G$;
 \item 
 $\p_i$ is a minimal $G_{i-1}$-admissible covering projection and $G_{i}$ is the lift of $G_{i-1}$ along $\p_i$
  for every $i\in \{1, \ldots, k\}$.
\end{enumerate}

It is now clear that Step (3) can be completed as follows: For each admissible pair $(\Gamma,G) \in \cP_0$, find
all minimal $G$-admissible elementary abelian covering projections $\p \colon \tilde{\Gamma} \to \Gamma$,
and the corresponding lifts $\tilde{G}$ of $G$, satisfying the condition $|\V(\tilde{\Gamma})| \le M$.
Denote the set of all thus obtained pairs $(\tilde{\Gamma},\tilde{G})$
by $\cP_1$ and think of it as forming the first level of the computations.
It is obvious that only those admissible pairs $(\Gamma,G) \in \cP_0$ had to be considered, for which
$|\V(\Gamma)| \le \frac{M}{2}$. In particular, if $M=10752$, then only the first $12$ pairs from Table~\ref{table:base}
have to be considered as base pairs. This procedure is then repeated recursively for $i=1,2,\ldots$, by applying it to the $i$-th level $\cP_i$ (instead of $\cP_{0}$) and yielding the set $\cP_{i+1}$ (instead of $\cP_1)$, until all the graphs in $\cP_i$ have order larger than $\frac{M}{2}$.
The union of all $\cP_i$, $i\ge 1$, then forms the set  $\tilde{\cP}$ defined in the description of Step (3).

If $M=10752$, then this procedure terminates at level $8$ and results in the set $\tilde{\cP}$ of $760$ relevant pairs.
To give the reader an idea of the number of covers involved together with computation times, we present some statistics in Table~\ref{table:level}. 
These were all run on a 2.93 GHz Quad-Core $\textrm{Intel}^{\circledR}$ $\textrm{Xeon}^{\circledR}$ processor X7350 at the Faculty of Mathematics and Physics, University of Ljubljana.
For each of the $12$ pairs $(\Gamma,G)$ from $\cP_0$ satisfying $|\V(\Gamma)| \le 5381$,
 each row of Table~\ref{table:level} displays 
the number of solvable regular covering projections obtained on each level during the computations and the computation time, respectively.
\begin{table}[!ht]\footnotesize
\caption{The number of covers obtained on each level together with computation times.} 
\centering 
\begin{tabular}{c  c  c  c  c  c  c  c  c  c  c} 
\hline\hline 
\\[-2ex]
ID &  Lev.~$1$ & Lev.~$2$ & Lev. $3$ & Lev. $4$ & Lev. $5$  & Lev. $6$ & Lev. $7$ & Lev. $8$ & Time\\ [0.5ex] 
\hline 
\\[-2ex]
1 &  56 & 90  & 75 & 34 & 15 & 7 & 2 & 1 & 62 hrs 27 mins\\
2 &  31 &  51 & 43 & 20 & 9 & 3  & - & - & 1 hr 54 mins \\
3 &  33 &  75 &  73 & 34 & 15 & 5& - & - & 2 hrs 43 mins \\
4 &  11 &  13 & 7 & 2 & 1 & - & - & - & 31 mins \\
5 &  6 &  6 & 2 & - & - & - & - & - & 8 mins \\
6 &  6 &  5 & 2 & - & - & - & - & - & 6 mins \\
7 &  4 &  2 & - & - & - & - & - & - & 1 min 27 secs\\
8 &  4 &  2 & - & - & - & - & - & - & 1 min 26 secs\\
9 &  3 &  1 & - & - & - & - & - & - & 34 secs\\
10 &  3 &  1 & - & - & - & - & - & - &  38 secs\\
11 &  3 &  1 & - & - & - & - & - & - &  38 secs\\
12 &  3 &  - & - & - & - & - & - & - & 1 min 2 secs\\
\hline
\end{tabular}
\label{table:level} 
\end{table}

Putting all of this together, we obtain $776$ relevant pairs with graph-order at most $10752$.
Furthermore, considering graphs up to isomorphism we have the following result.

\begin{theorem}
\label{the:census}
There are precisely $564$ pairwise non-isomorphic connected tetravalent graphs  of order at most $10752$
admitting a $\half$-arc-transitive group $G$ with $G_v$ isomorphic to $\D_4$.
\end{theorem}

For each of these $564$ graphs, the approach ensures existence of a $\half$-arc-transitive group of
automorphisms $G$ with $G_v \cong \D_4$. However, the full automorphism group of these graphs is in all but two cases
larger than this guaranteed $G$ and is in fact dart-transitive. The two exceptional graphs, where the $\half$-arc-transitive group
$G$ equals the full automorphism group, have order $10752$ and are precisely the two graphs constructed in
\cite{ConMar} and \cite{ConPotSpa}.
This confirms Theorem~\ref{the:main}.

The data about graphs of our census is available on-line at \cite{PP14}. The package contains three files. The ``Census-GHAT-10752-Gv8.mgm"  file  contains
{\sc Magma} code that generates the corresponding graphs. To load the contents of this file  into {\sc Magma}, the file ``Census-GHAT-10752-Gv8.txt" (containing the list of neighbours of each graph) is needed. Additionally, the ``Census-GHAT-10752-Gv8.csv" file is a ``comma separated values" file representing a spreadsheet containing some precomputed graph invariants. Each line of this file represents one of the graphs in the census, and has four fields as follows:
\begin{enumerate}
\item[(i)] $\mathtt{ID:}$ the ID number of the graph;
\item[(ii)] $\mathtt{|V|:}$ the order of the graph;
\item[(iii)] $\mathtt{|A_v|}:$ the order of the vertex-stabiliser in the automorphism group of the graph;
\item[(iv)] $\mathtt{AT:}$ this field contains ``true" if the graph is arc-transitive and ``false" otherwise.
\end{enumerate}



\end{document}